\documentclass{amsart}
\usepackage{latexsym}
\usepackage{amstext}
\usepackage{amsmath}
\usepackage{amssymb}
\usepackage{amsthm}
\usepackage{url}

\newtheorem{theorem}{Theorem}
\newtheorem{lemma}[theorem]{Lemma}

\newtheorem{proposition}[theorem]{Proposition}

\theoremstyle{definition}

\DeclareMathOperator{\ord}{ord}
\def\mod#1{{\ifmmode\text{\rm\ (mod~$#1$)}
\else\discretionary{}{}{\hbox{ }}\rm(mod~$#1$)\fi}}

\begin{document}
\title{The Nagell-Ljunggren equation via Runge's method}

\author{Michael A. Bennett}
\address{Department of Mathematics, University of British Columbia, Vancouver, BC, Canada}
\email{bennett@math.ubc.ca}
\thanks{Supported in part by a grant from NSERC}

\author{Aaron Levin}
\address{Department of Mathematics, Michigan State University, East Lansing, MI, U.S.A.}
\email{adlevin@math.msu.edu}
\thanks{Supported in part by NSF grant DMS-1102563}

\date{\today}

\begin{abstract}
The Diophantine equation $\frac{x^n-1}{x-1}=y^q$ has four known solutions in integers $x, y, q$ and $n$ with $|x|, |y|, q > 1$ and $n > 2$. Whilst we expect that there are, in fact,  no more solutions, such a result is well beyond current technology. In this paper, we prove that if $(x,y,n,q)$ is a solution to this equation, then $n$ has three or fewer prime divisors, counted with multiplicity.  This improves a result of Bugeaud and Mih{\u{a}}ilescu.
\end{abstract}

\maketitle

\section{Introduction}

The {\it Nagell-Ljunggren equation}
\begin{align}
\label{NL}
\frac{x^n-1}{x-1}=y^q, \quad \text{in integers $|x|>1, |y|>1, n>2, q\geq 2$}
\end{align}
arises in a wide variety of contexts, ranging from  group theory \cite{Gu} to irrationality criteria \cite{Sa}; for an excellent survey, the reader is directed to Bugeaud and Mignotte \cite{BuMi}.
Equation (\ref{NL}) attracted attention initially due to its connection to Catalan's conjecture and the corresponding Diophantine equation
\begin{equation} \label{cat}
x^n-y^q=1.
\end{equation}
Whilst Catalan's conjecture was proven by Mih{\u{a}}ilescu \cite{Mih2} in 2004, the Nagell-Ljunggren equation has, in a certain sense, outlived its more illustrious cousin, in that it remains unknown to date whether the number of solutions to (\ref{NL}) in the four variables $x, y, n$ and $q$ is finite. Indeed, such a conclusion is beyond current technology even in the restricted case where $n=q$. Equation  \eqref{NL}  does have precisely four known solutions:
\begin{align}
\label{sols}
\frac{3^5-1}{3-1}=11^2, \quad \frac{7^4-1}{7-1}=20^2, \quad \frac{18^3-1}{18-1}=7^3  \quad \mbox{ and }  \quad 
\frac{(-19)^3-1}{(-19)-1} = 7^3 \tag{$\mathcal{NL}$}
\end{align}
and there is an impressively large literature providing constraints upon any hitherto unknown ones. In particular, combining results from \cite{BuMi2}, \cite{BMR},  \cite{Lj}, \cite{Mih2},  \cite{Na1} and \cite{Na2}, we have
\begin{proposition} 
\label{summ}
If $(x,y,n,q)$ is a solution of equation \eqref{NL} not in \eqref{sols}, then 
\begin{itemize}
\item $q \geq 3$ is odd,
\item The least prime divisor $p$ of $n$ satisfies $p \geq 5$,
\item $|x| \geq 10^4$ and $x$ has a prime divisor $p \equiv 1 \mod{q}$.
\end{itemize}
\end{proposition}
We can in fact say rather more if we additionally assume that $x$ is positive, improving the lower bounds on both $|x|$ and the least prime divisor of $n$ in this case (see \cite{BHM}, \cite{BuMi0}); though similar arguments can be applied to negative values of $x$ to sharpen Proposition \ref{summ}, we will not have need of such a result here.

One feature distinguishing equation (\ref{NL}) from (\ref{cat}) is that the reduction to the case where $n$ is prime is without loss of generality in (\ref{cat}), but not in (\ref{NL}). That being said, there is some degree of control over how composite the exponent $n$ can be. Specifically, writing $\omega (n)$ for the number of distinct prime divisors of $n$, an argument of Shorey \cite{Sho}, together with the results of \cite{Benn}, implies that $\omega (n) \leq q-2$ (for $x$ both positive and negative, though this is stated in \cite{Sho} only for the former).  In
\cite{BM}, Bugeaud and Mih{\u{a}}ilescu sharpened this substantially, proving that if $(x,y,n,q)$ is a solution to \eqref{NL}, with $x > 1$, then necessarily 
$$
\omega(n)\leq \Omega(n)\leq 4,
$$
where $\Omega (n)$ is the total number of prime divisors of $n$, counted with multiplicity. 
Our main result is the following improvement of this:

\begin{theorem}
\label{mtheorem}
If $(x,y,n,q)$ is a solution of \eqref{NL}, then $\omega(n)\leq \Omega (n) \leq  3$.
\end{theorem}

The proof of this theorem relies upon a careful application of the classical method of Runge to an  equation of the shape $f(x)=g(y)$, where $f$ and $g$ are polynomials with integer coefficients. The result of Bugeaud and Mih{\u{a}}ilescu \cite{BM} depends fundamentally upon earlier work of Mih{\u{a}}ilescu \cite{Mih}, which one might view, in part, as an application of Runge's method over cyclotomic fields. The proof of the crucial inequality from \cite{Mih}, given in Theorem  2 of that paper, uses tools from cyclotomic field theory and techniques developed in the proof of Catalan's conjecture.  What we prove, in an essentially elementary fashion and independent of the results of \cite{Mih}, is that  solutions to \eqref{NL} necessarily satisfy $\omega(n)\leq \Omega (n) \leq  3$, with the possible exception of the special case where $(n,q)=(p^5,5)$ and $p \mid x-1$, for $p$ prime.  We eliminate this remaining exceptional case by appealing to \cite{Mih}.

\vskip1.5ex
Our main tools in proving Theorem \ref{mtheorem}  are the following pair of results.

\begin{theorem} \label{good}
Suppose that $q, r, s, x, y$ and $z$ are integers with $q$ prime and  $3 \leq r< s$, such that 
\begin{equation} \label{eq1}
y^q=\frac{x^r-1}{x-1}
\end{equation}
and
\begin{equation} \label{eq2}
z^q=\frac{x^s-1}{x-1}.
\end{equation}
Then
\begin{align}
\label{funineq}
|x| < 4 \,  \max \left\{ 1, \frac{s}{qr}  \right\}  \; q^{r-1 + \frac{s-r}{q-1}}.
\end{align}
\end{theorem}

\vskip2ex
\begin{theorem} \label{good2}
Suppose that  integers $x, y, z$ and $u$ and odd primes $p$ and $q$ satisfy
\begin{equation*}
p y^q=\frac{x^p-1}{x-1},
\end{equation*}
\begin{equation*}
p^2 z^q=\frac{x^{p^2}-1}{x-1}
\end{equation*}
and
\begin{equation*}
p^3 u^q=\frac{x^{p^3}-1}{x-1}.
\end{equation*}
Then
\begin{align}
\label{funineq3}
|x| <17  \; p^3 \left( \max \left\{ 1, \frac{p}{q} \right\} \right)^3 \; q^{p^2+p-2+(p^3-1)/(q-1)}.
\end{align}
\end{theorem}

\section{Preliminaries}

Before proceeding with the proof of Theorem \ref{mtheorem}, we require a pair of lemmata of a combinatorial nature. Given solutions to (\ref{NL}), these will enable us to deduce the existence of simultaneous solutions to a number of related equations. 
The first is a result of Shorey \cite[Lemma 7]{Sho}.  For a positive integer $n$, let $Q_n=\phi(G(n))$, where $G(n)$ is the square-free part of $n$ and $\phi$ is Euler's function.

\begin{lemma}
\label{Sholem}
Let $(x,y,n,q)$ be a solution of \eqref{NL} with $n$ odd.  If the divisor $D$ of $n$ satisfies $(D,n/D)=(D,Q_{n/D})=1$, then there exist integers $y_1$ and $y_2$ with $y_1y_2=y$ and
\begin{align*}
y_1^q&=\frac{(x^D)^{n/D}-1}{x^D-1},\\
y_2^q&=\frac{x^D-1}{x-1}.
\end{align*}
\end{lemma}

We use this to prove the following 
\begin{lemma}
\label{bonzo}
Let $(x,y,n,q)$ be a solution of \eqref{NL} with $n$ odd and write
$$
n = p_1^{\alpha_1} \cdots p_k^{\alpha_k},
$$
where $\alpha_i \in \mathbb{N}$ and the $p_i$  are primes with
$$
p_1 < p_2 < \cdots < p_k.
$$
Then, for each $1 \leq s \leq k$ and $0 \leq t \leq \alpha_s-1$, writing
$$
M_{s,t} = p_s^{t} \, p_{s+1}^{\alpha_{s+1}} \cdots p_k^{\alpha_k},
$$
there exists an integer $y_{s,t}$ and $\delta_s \in \{ 0, 1 \}$ such that we have
\begin{equation} \label{lump}
\frac{x^{p_s M_{s,t}}-1}{x^{M_{s,t}}-1} = p_s^{\delta_s} \,  y_{s,t}^q. \; \; 
\end{equation}
If $\delta_s =1$, then $q \mid \alpha_s$.

\end{lemma}

\begin{proof}

Observe that, by assumption,
\begin{equation} \label{munch}
\prod_{s=1}^k \prod_{t=0}^{\alpha_s-1} \frac{x^{p_s M_{s,t}}-1}{x^{M_{s,t}}-1} =   y^q.
\end{equation}
Let us suppose first that $k=1$.
Since, given integers $|z|>1$ and $m \geq 1$, we have that
$$
\gcd \left( \frac{z^m-1}{z-1}, z-1 \right) \; \mbox{ divides } \;  m,
$$
it follows, writing $X=x^{M_{1,t}}=x^{p_1^t}$ for $t \in \{0, 1, \ldots, \alpha_1-1 \}$, that
$$
\gcd \left( \frac{X^{p_1}-1}{X-1}, X-1 \right) \in \{ 1, p_1 \}.
$$  
Moreover, if $X \equiv 1 \mod{p_1}$, say $X-1=p_1a$ with $a\in\mathbb{Z}$, then
\begin{align*}
X^{p_1-1}+X^{p_1-2}+\cdots +1&= (1+p_1a)^{p_1-1}+(1+p_1a)^{p_1-2}+\cdots +1\\
&\equiv p_1+\sum_{i=1}^{p_1-1}ip_1a\mod{p_1^2}\\
&\equiv p_1+a(p_1-1)p_1^2/2 \equiv p_1 \mod{p_1^2},\\
\end{align*}
and, in particular, repeatedly applying Fermat's Little Theorem, we have 
$$
\ord_{p_1} \left(  \frac{x^{p_1^{t+1}}-1}{x^{p_1^t}-1} \right)  =
\left\{
\begin{array}{cl}
 1 & \mbox{ if } x \equiv 1 \mod{p_1} \\
 0 & \mbox{ otherwise,} \\
 \end{array}
 \right.
$$
for each $0 \leq t \leq \alpha_1-1$. There thus exist integers $y_t$, again for $0 \leq t \leq \alpha_1-1$, such that
$$
 \frac{x^{p_1^{t+1}}-1}{x^{p_1^t}-1} = p_1^\delta \, y_t^q,
 $$
 where 
 $$
 \delta = 
 \left\{
\begin{array}{cl}
 1 & \mbox{ if } x \equiv 1 \mod{p_1} \\
 0 & \mbox{ otherwise.} \\
 \end{array}
 \right.
$$
Since, from (\ref{munch}), we have
$$
\alpha_1 \delta = \ord_{p_1}  \prod_{t=0}^{\alpha_1-1} \frac{x^{p_1^{t+1}}-1}{x^{p_1^t}-1} = q \ord_{p_1} y  \equiv 0 \mod{q},
$$
the desired result follows.

Suppose now that the stated conclusion is true for solutions to \eqref{NL} with $\omega(n)=k-1$, where $k \geq 2$. Assuming that $(x,y,n,q)$ is a solution of \eqref{NL} with $n$ odd and $\omega (n)=k$, say,
$$
n = p_1^{\alpha_1} \cdots p_k^{\alpha_k},
$$
we may thus apply Lemma \ref{Sholem} with $D= n/p_1^{\alpha_1}$ to deduce the existence of  integers $Z$ and $W$  such that
\begin{equation} \label{rup}
\frac{X-1}{x-1} = Z^q \; \mbox{ and } \; \frac{X^{p_1^{\alpha_1}}-1}{X-1} = W^q,
\end{equation}
where $X=x^{n/p_1^{\alpha_1}}$.
As previously, we find that
$$
q \ord_{p_1} W = \ord_{p_1}  \prod_{t=0}^{\alpha_1-1} \frac{X^{p_1^{t+1}}-1}{X^{p_1^t}-1}  
=
\left\{
\begin{array}{cl}
 \alpha_1 & \mbox{ if } X \equiv 1 \mod{p_1} \\
 0 & \mbox{ otherwise,} \\
 \end{array}
 \right.
$$
whereby we have (\ref{lump}) for $s=1$. Applying our inductive hypothesis to the first equation in (\ref{rup}) (where we have $\omega (n/p_1^{\alpha_1})=k-1$), we conclude as desired.
\end{proof}

\section{Proof of Theorem \ref{good}}

The idea of the proof is to apply Runge's method to an appropriate curve.  The classical Runge's method is used to bound integer points on superelliptic curves of the form $y^q=f(x)$, where $f\in\mathbb{Z}[x]$ is monic and $q \mid \deg f$.  In order to construct a situation where this version of Runge's method applies, we will consider a product of powers of $\frac{x^r-1}{x-1}$ and $\frac{x^s-1}{x-1}$ to obtain a polynomial $f$ of degree divisible by $q$ and then examine integer solutions to $w^q=f(x)$.

Let $q,r,s,x,y,z\in\mathbb{Z}$ be a solution to \eqref{eq1} and \eqref{eq2}, as in the statement of the theorem.  We may assume throughout that $|x|>2$. Note that the main result of Bennett \cite{Benn} implies that the Diophantine equation
$$
x^{t} X^q - (x-1) Y^q = 1
$$
has, if $t=1$, precisely the solution $X=Y=1$ in nonzero integers $X$ and $Y$, and, for a fixed positive integer $2 \leq t \leq q-1$, at most a single nonzero solution $X, Y$. Choosing integers $0 \leq r_0, s_0 < q$ such that $r \equiv r_0 \mod{q}$ and $s \equiv s_0 \mod{q}$ and rewriting (\ref{eq1}) and (\ref{eq2}) as
$$
x^{r_0} \left( x^{(r-r_0)/q} \right)^q - (x-1) y^q = 1 \; \mbox{ and } \;
x^{s_0} \left( x^{(s-s_0)/q} \right)^q - (x-1) z^q = 1,
$$
we may thus conclude that $1, r$ and $s$ are pairwise incongruent modulo $q$ (whereby, since Proposition \ref{summ} implies  that $3 \nmid rs$, $q \geq 5$). In particular, since $r \not\equiv 1 \mod{q}$ and $s  \not\equiv 1 \mod{q}$, we may define an integer $a$ with $0 < a < q$, via
$$
a \equiv -(r-1)^{-1}(s-1)\mod{q}.
$$
Since $r \not\equiv s \mod{q}$, it follows that $a \neq q-1$, whence $1 \leq a \leq q-2$. Let us now define
$$
N= a(r-1) +s-1,
$$
so that
$$
N \equiv  -(r-1)^{-1}(s-1) (r-1) + (s-1) \equiv 0 \mod{q},
$$
whereby $N/q$ is an integer.

 We consider the equation
\begin{align}
\label{weq}
w^q=\left(\frac{x^r-1}{x-1}\right)^a\left(\frac{x^s-1}{x-1}\right),
\end{align}
where now Runge's condition is satisfied.
This equation has the integer solution $(w,x)=(y^az,x)$.  Consider the Laurent series expansion
\begin{align*}
&\left[\left(\frac{x^r-1}{x-1}\right)^a\left(\frac{x^s-1}{x-1}\right)\right]^{\frac{1}{q}}=(x^r-1)^{a/q}(x^s-1)^{1/q}(x-1)^{(-a-1)/q}\\
&=x^{N/q}\left(1-\frac{1}{x^r}\right)^{a/q}\left(1-\frac{1}{x^s}\right)^{1/q}\left(1-\frac{1}{x}\right)^{(-a-1)/q}\\
&=x^{N/q}\left(\sum_{i=0}^\infty (-1)^i\binom{a/q}{i}x^{-ri}\right)\left(\sum_{j=0}^\infty (-1)^j\binom{1/q}{j}x^{-sj}\right)\left(\sum_{k=0}^\infty (-1)^k\binom{-(a+1)/q}{k}x^{-k}\right)\\
&=x^{N/q}\sum_{n=0}^\infty a_nx^{-n},
\end{align*}
where 
\begin{align*}
a_n=\sum_{\substack{ri+sj+k=n\\i,j,k\geq 0}}\binom{a/q}{i}\binom{1/q}{j}\binom{-(a+1)/q}{k}(-1)^{i+j+k}.
\end{align*}
Note that the series converges if $|x|>1$.  It follows that, for any solution $(w,x)$ of \eqref{weq} with $|x|>1$,
\begin{align*}
w=\zeta x^{N/q}\sum_{n=0}^\infty a_nx^{-n},
\end{align*}
for some $q$th root of unity $\zeta$.  In our case, $w$ and $x$ are both real, so we must have
\begin{equation} \label{bubs}
w=\pm x^{N/q}\sum_{n=0}^\infty a_nx^{-n}.
\end{equation}
Before we proceed further, we need to understand the coefficients $a_n$ somewhat better.

\begin{lemma}
\label{vlem}
Let $n$ be a nonnegative integer. Then we have
\begin{align}
\ord_q a_n&=-n-\ord_q(n!),\label{ordan}\\
 q^{n+\left[\frac{n}{q-1}\right]}a_n&\in \mathbb{Z},\label{one}
\end{align}
and
\begin{equation} \label{two}
|a_n|\leq (\left[ n/r\right]+1)(\left[ n/s\right]+1).
\end{equation}
\end{lemma}
\begin{proof}
If $l$ is a nonnegative integer and $m$ is an integer, coprime to  a given prime $q$, then we have
$$
\ord_q \binom{m/q}{l} = -l - \ord_q (l!).
$$
Let $n\geq 0$ and suppose that $i, j$ and $k$ are nonnegative integers  for which $ri+sj+k=n$.  We thus have
$$
\ord_q \left(\binom{a/q}{i}\binom{1/q}{j}\binom{-(a+1)/q}{k}\right)=
-i-j-k - \ord_q (i! j! k! ).
$$
 It follows, if $i$ and $j$ are nonnegative integers, not both zero, that
 $$
 \begin{array}{c}
- \ord_q \binom{-(a+1)/q}{n} + \ord_q \left(\binom{a/q}{i}\binom{1/q}{j}\binom{-(a+1)/q}{k}\right) \\
 = (r-1)i + (s-1) j + \ord_q \left( \frac{(ri+sj+k)!}{i!j!k!} \right) \geq r-1,
 \end{array}
 $$
 since $\frac{(ri+sj+k)!}{i!j!k!}$ is an integer. We thus have
 
\begin{equation*}
\ord_q a_n=\ord_q \left(\binom{-(a+1)/q}{n}(-1)^{n}\right)=-n-\ord_q(n!),
\end{equation*}
whereby
$$
q^{n+ \ord_q (n!)} a_n \in \mathbb{Z}.
$$
Statement (\ref{one}) follows upon observing that
$$
\ord_q(n!)=\sum_{k=1}^\infty \left[\frac{n}{q^k}\right]\leq \left[\frac{n}{q-1}\right].
$$

To prove inequality (\ref{two}), recall that $1 \leq a \leq q-2$ and apply the easy fact that $|c|\leq q$ implies $|\binom{c/q}{i}|\leq 1$ for all $i\geq 0$. 
\end{proof}

Let us now define
$$
P(x)=q^{N/q + \left[\frac{N}{q(q-1)}\right]}x^{N/q}\sum_{n=0}^{N/q} a_nx^{-n}.
$$
Note that, by Lemma \ref{vlem}, $P(x) \in\mathbb{Z}[x]$ and, since $\ord_q a_{n+1}<\ord_q a_n\leq 0$ for all $n$,
$$
q^{N/q+ \left[\frac{N}{q(q-1)}\right]}w \mp P(x) \neq 0.
$$

Let $m=N/q+1$.  Then from our definitions, Lemma \ref{vlem} and (\ref{bubs}),
$$
\begin{array}{l}
\left|q^{N/q+\left[\frac{N}{q(q-1)}\right]}w \mp P(x)\right|= q^{N/q+\left[\frac{N}{q(q-1)}\right]}\left|x^{N/q}\sum_{n=m}^{\infty} a_nx^{-n}\right|\\ \\
\leq q^{N/(q-1)}|x|^{N/q}\sum_{n=m}^{\infty} (\left[ n/r\right]+1)(\left[ n/s\right]+1)|x|^{-n}. \\
\end{array}
$$
Since we have that
$$
\frac{(\left[ (n+1)/r\right]+1)(\left[ (n+1)/s\right]+1)}{(\left[ n/r\right]+1)(\left[ n/s\right]+1)} \leq 3,
$$
for $n \geq m \geq 2$ and $s > r \geq 3$, where the inequality is sharp (corresponding to $(r,s,n)=(r,2r,2r-1)$),
it follows that the positive integer $\left|q^{N/q+\left[\frac{N}{q(q-1)}\right]}w \mp P(x)\right|$ is bounded above by
$$
\begin{array}{l}
q^{N/(q-1)}([m/r]+1)([m/s]+1)|x|^{-1}\sum_{n=0}^{\infty} 3^n|x|^{-n}\\ \\
=q^{N/(q-1)}([m/r]+1)([m/s]+1)\frac{1}{|x|-3}
\end{array}
$$
and hence
\begin{align*}
|x|<3+q^{N/(q-1)}([m/r]+1)([m/s]+1).
\end{align*}
From
$$
N= a(r-1) +s-1 \leq (q-2)(r-1)+s-1,
$$
we find that 
$$
\frac{m}{r} <1+\frac{s}{qr} \; \mbox{ and } \; 
\frac{m}{s}<1,
$$
whereby the result follows.

\section{Proof of Theorem \ref{good2}}

Let $x,y,z,u,p$ and $q$ be as in the statement of Theorem \ref{good2}.  We choose $a$ and $b$ to be the smallest nonnegative integers satisfying 
$$
a \equiv 2p+1 \mod{q} \; \mbox{ and } \; b \equiv -p-2 \mod{q}.
$$
Then we have, writing $N=a(p-1)+b(p^2-1)+(p^3-1)$, that
$$
N \equiv  a+ 2b + 3 \equiv 0 \mod{q}
$$
and hence
\begin{equation} \label{herc}
\left(\frac{x^p-1}{x-1}\right)^a \, \left(\frac{x^{p^2}-1}{x-1}\right)^b \left(\frac{x^{p^3}-1}{x-1}\right) = w^q,
\end{equation}
for some integer $w$, where the left hand side is a polynomial in $x$ with degree divisible by $q$.
We have 
\begin{align*}
& \left[\left(\frac{x^p-1}{x-1}\right)^a \, \left(\frac{x^{p^2}-1}{x-1}\right)^b \left(\frac{x^{p^3}-1}{x-1}\right) \right]^{\frac{1}{q}} \\
&=(x^p -1)^{a/q}(x^{p^2}-1)^{b/q}(x^{p^3}-1)^{1/q} (x-1)^{(-a-b-1)/q}\\
&=x^{N/q} \left(1-\frac{1}{x^p}\right)^{a/q} \left(1-\frac{1}{x^{p^2}}\right)^{b/q} \left(1-\frac{1}{x^{p^3}}\right)^{1/q}\left(1-\frac{1}{x}\right)^{(-a-b-1)/q}\\
&=x^{N/q}\sum_{n=0}^\infty b_nx^{-n},
\end{align*}
where $\sum_{n=0}^\infty b_nx^{-n}$ is the product of the four series
$$
\sum_{i=0}^\infty (-1)^i\binom{a/q}{i}x^{-pi}, \; \; \sum_{j=0}^\infty (-1)^j\binom{b/q}{j}x^{-p^2j}, \;  \; 
\sum_{k=0}^\infty (-1)^k\binom{1/q}{k}x^{-p^3k}
$$
and
$$
\sum_{l=0}^\infty(-1)^l \binom{-(a+b+1)/q}{l}x^{-l}.
$$
We may thus write
\begin{align*}
b_n=\sum_{\substack{pi+p^2j+p^3k+l=n\\i,j,k,l\geq 0}}   \binom{a/q}{i}
\binom{b/q}{j} \binom{1/q}{k}\ \binom{-(a+b+1)/q}{l}(-1)^{i+j+k+l}.
\end{align*}
Note here that we have $0 \leq a, b \leq q-1$, so that $a+b+1 \leq 2q-1$. It follows that
$$
\left| \binom{-(a+b+1)/q}{l} \right| < l+1
$$
and hence
\begin{equation} \label{bee-big}
|b_n| < \left( \left[ \frac{n}{p} \right] +1 \right)  \left( \left[ \frac{n}{p^2} \right] +1 \right)  \left( \left[ \frac{n}{p^3} \right] +1 \right)
(n+1).
\end{equation}
The series $\sum_{n=0}^\infty b_nx^{-n}$ thus converges if $|x|>1$ and hence, for any real solution $(w,x)$ of \eqref{herc} with $|x|>1$, we can write
\begin{equation*}
w=\pm x^{N/q}\sum_{n=0}^\infty b_nx^{-n}.
\end{equation*} 
Arguing as in the proof of Lemma \ref{vlem}, we find that
$$
\ord_q \left(\binom{a/q}{i}
\binom{b/q}{j} \binom{1/q}{k}\ \binom{-(a+b+1)/q}{l}\right)\geq
-i-j-k -l - \ord_q (i! j! k! l!).
$$
We assume now that $q\nmid (a+b+1)$ (or equivalently, $p\neq q$).  It will be clear from our argument that one obtains even stronger bounds when $q\mid (a+b+1)$ (note that in this case, $(x^p-1)^a(x^{p^2}-1)^b(x^{p^3}-1)$ is a perfect $q$th power).  Then, assuming $q\nmid (a+b+1)$, after a little work, we conclude that
\begin{equation*}
\ord_q b_n=\ord_q \binom{-(a+b+1)/q}{n}=-n-\ord_q(n!),
\end{equation*}
whence $b_n \neq 0$ for each $n$ and 
$$
q^{n+ \ord_q (n!)} b_n \in \mathbb{Z}.
$$
We thus have
\begin{equation*}
q^{n+\left[\frac{n}{q-1}\right]}b_n\in \mathbb{Z},
\end{equation*}
for each $n \geq 0$. Define
$$
P(x)=q^{N/q+ \left[\frac{N}{q(q-1)}\right]}x^{N/q}\sum_{n=0}^{N/q} b_nx^{-n}.
$$
We have that $P(x) \in\mathbb{Z}[x]$ and, since $\ord_q b_{n+1}<\ord_q b_n\leq 0$ for all $n$,
$$
q^{N/q+ \left[\frac{N}{q(q-1)}\right]}w \mp P(x)\neq 0.
$$

Setting as before $m=N/q+1$, we thus have  that
$$
|q^{N/q+\left[\frac{N}{q(q-1)}\right]}w \mp P(x)| = q^{N/q+\left[\frac{N}{q(q-1)}\right]}\left|x^{N/q}\sum_{n=m}^{\infty} b_nx^{-n}\right|
$$
is a positive integer. On the other hand, via (\ref{bee-big}), the right hand side here is bounded above by
$$
q^{N/(q-1)}|x|^{N/q}\sum_{n=m}^{\infty} (\left[ n/p\right]+1)(\left[ n/p^2\right]+1) (\left[ n/p^3\right]+1) (n+1) |x|^{-n}
$$
which is in turn at most
$$
q^{N/(q-1)}([m/p]+1)([m/p^2]+1)([m/p^3]+1)(m+1) |x|^{-1}\sum_{n=0}^{\infty} (7/2)^n|x|^{-n}.
$$
Here, we have used that
$$
\frac{(\left[ (n+1)/p\right]+1)(\left[ (n+1)/p^2\right]+1) (\left[ (n+1)/p^3\right]+1) (n+2)}{(\left[ n/p\right]+1)(\left[ n/p^2\right]+1) (\left[ n/p^3\right]+1) (n+1)} \leq \frac{2240}{729} < 7/2,
$$
with equality in the first of these inequalities corresponding to $(n,p)=(26,3)$.
It follows that
$$
1 \leq q^{N/(q-1)}([m/p]+1)([m/p^2]+1)([m/p^3]+1)(m+1) \frac{1}{|x|-7/2},
$$
whereby
\begin{align*}
|x|<7/2+q^{N/(q-1)}([m/p]+1)([m/p^2]+1)([m/p^3]+1)(m+1).
\end{align*}
If we now use the fact that 
$$
N \leq (q-1)(p-1)+(q-1)(p^2-1)+p^3-1 = (q-1) (p^2+p-2) + p^3-1,
$$
we find that $m<p^3$ and
\begin{align*}
m<
\begin{cases}
2p^2+p-1 \quad &\text{if $q\geq p$},\\
2 p^3/q-1 \quad &\text{if $q<p$}.
\end{cases}
\end{align*}

We conclude easily that
$$
|x| < \frac{7}{2}+ \frac{49}{3}p^3\left(\max\left\{1, \frac{p}{q} \right\}\right)^3q^{p^2+p-2+(p^3-1)/(q-1)},
$$
and the result follows.

\section{Proof of Theorem \ref{mtheorem}}

We now prove Theorem \ref{mtheorem}. Let $(x,y,n,q)$ be a solution of \eqref{NL} with $n$ odd and $\Omega(n)\geq 4$.  
As in Lemma \ref{bonzo}, we write
$$
n = p_1^{\alpha_1} \cdots p_k^{\alpha_k} 
$$
where $\alpha_i \in \mathbb{N}$ and the $p_i$ are primes with
$$
p_1 < p_2 < \cdots < p_k.
$$
Note that $k=\omega(n)$, $\sum_{i=1}^k \alpha_i=\Omega(n)$, and that, by assumption, equation (\ref{munch}) holds.

Define $\delta_s$ as in Lemma \ref{bonzo}, for $1 \leq s \leq k$. We begin by supposing that $\delta_s=0$ for each value of $s$. Let us set 
$$
p = 
\left\{
\begin{array}{cc}
p_1 & \mbox{ if } \alpha_1 \geq 2, \\
p_2 & \mbox{ if } \alpha_1 = 1, \\
\end{array}
\right.
$$
and define the positive integer $m$ via $n=m p_1 p$.
From Lemma \ref{bonzo},  we may find integers $y_1$ and $y_2$ for which
$$
\frac{x^{n}-1}{x^{m}-1} = (y_1 y_2)^q \; \mbox{ and } \; 
 \frac{x^{mp}-1}{x^{m}-1} = y_2^q.
$$
Applying Theorem \ref{good} with $r=p$, $s=p_1p$, and using the fact that $|x| \geq 2q+1$, we thus have
$$
2^{m} q^{m}<
4 \,\, \max \left\{ 1, \frac{p_1}{q} \right\}  \; q^{p-1 + \frac{p (p_1-1)}{q-1}}.
$$
Since $q \geq 3$, $p \geq p_1 \geq 5$ and $\Omega (n) \geq 4$ (whence $m \geq p^2$),
it follows that
$$
2^{p^2} q^{p^2} \leq 2^m q^m <   p \, q^{p + \frac{p ( p -1)}{q-1}}
$$
and so
$$
2^{p^2} q^{p(p-1)(q-2)/(q-1)} <  p,
$$
an immediate contradiction.

 Next, let us suppose that $\delta_s=1$ for at least one  value of $s$. Choosing $s$ to be the smallest such index, we thus have that $q \mid \alpha_s$ and, writing 
$$
X=x^\beta \; \; \mbox{ where } \; \; \beta = \prod_{j > s} p_j^{\alpha_j}, 
$$
and applying Lemma \ref{bonzo},
$$
\frac{X^{p_s^{\alpha_s-i}}-1   }{ X  ^{p_s^{\alpha_s-i-1}}-1  } = p_s y_i^q, \; \; \mbox{ for } i = 0, 1, \ldots, \alpha_s-1
$$
and integers $y_i$. Setting $Z = X^{p_s^{\alpha_s-3}}$, we may thus write
$$
\frac{Z^{p_s^j}-1}{Z-1} = p_s^j z_j^q \; \mbox{ for } j=1, 2, 3,
$$
where the $z_j$ are integers. Theorem \ref{good2} thus implies that
\begin{equation} \label{fork}
\left| Z \right| < 17  \; p_s^3 \left( \max \left\{ 1, \frac{p_s}{q} \right\} \right)^3\, q^{p_s^2+p_s-2+(p_s^3-1)/(q-1)}.
\end{equation}
Since
$$
|Z| \geq |x|^{p_s^{\alpha_s-3}},
$$
and Proposition \ref{summ} enables us to conclude that $|x| \geq 2q+1$, we thus have
$$
2^{p_s^{\alpha_s-3}} \, q^{p_s^{\alpha_s-3}} < 17  \; p_s^3 \left( \max \left\{ 1, \frac{p_s}{q} \right\} \right)^3\, q^{p_s^2+p_s-2+(p_s^3-1)/(q-1)}.
$$
From $p_s \geq 5$ and the fact that $q \mid \alpha_s$, we reach an easy contradiction,
at least provided either $q \geq 7$, or $q \in \{ 3, 5 \}$ and $\alpha_s \neq q$. We may thus suppose that $(q,\alpha_s) = (3,3)$ or $(5,5)$ (whereby, from the fact that $\omega (n) \leq q-2$, we have $k=1$ and $k \leq 3$, respectively). We therefore have either 
$n=p^3$  for $p$ prime (so that $\Omega (n) < 4$ as desired) or
$(q,\alpha_s) = (5,5)$.
In the latter case, if $s < k$, then 
$$
|Z| =|x|^{ \beta p_s^{2}} > |x|^{ p_s^{3}} \geq 10^{4 p_s^{3}},
$$
contradicting (\ref{fork}). We may thus assume that $s=k$. We will show that necessarily $k=1$. Suppose otherwise. Then setting $X= x^{p_k^3}$ and applying Lemma \ref{bonzo}, we thus have
$$
\frac{X^{p_k} -1}{X-1} = p_k y_1^5,
$$
$$
\frac{X^{p_k^2} -1}{X-1} = p_k^2 y_2^5
$$
and
$$
\frac{X^{p_{k-1} p_k^2} -1}{X-1} = p_k^2 y_3^5,
$$
for integers $y_1, y_2$ and $y_3$.  We choose integers $a$ and $b$, with $0 \leq a, b \leq 4$,  such that
$$
a (p_k-1) + b (p_k^2-1) +  p_{k-1} p_k^2-1 \equiv 0 \mod{5}
$$
and
$$
a+2b+2 \equiv 0 \mod{5}.
$$
This is always possible as $\det  \left(
 \begin{matrix}
  p_k-1 & p_k^2-1\\
  1 & 2 
 \end{matrix}
\right)=-(p_k-1)^2
$ and $p_k\not\equiv 1\pmod{5}$ (say by combining the proof of Lemma 9 of \cite{Sho} with Theorem 1.1 of \cite{Benn}).

We now argue as in the proof of Theorem \ref{good2}, considering the Laurent series expansion
$$
\left[\left(\frac{X^{p_k}-1}{X-1}\right)^a \, \left(\frac{X^{p_k^2}-1}{X-1}\right)^b \left(\frac{X^{p_{k-1}p_k^2}-1}{X-1}\right) \right]^{1/5} = \, X^{N/5}\sum_{n=0}^\infty c_nX^{-n},
$$
where now
$$
N=a (p_k-1) + b (p_k^2-1) +  p_{k-1} p_k^2-1.
$$
Here, we have
\begin{align*}
c_n=\sum  \binom{a/5}{i}
\binom{b/5}{j} \binom{1/5}{l}\ \binom{-(a+b+1)/5}{m}(-1)^{i+j+l+m},
\end{align*}
where the sum is over nonnegative integers $i, j, l$ and $m$ with 
$$
p_ki+p_k^2 j + p_{k-1} p_k^2 l + m = n.
$$
We note that $a+b+1\equiv 0\pmod{5}$ easily implies that either $p_k\equiv 0\pmod{5}$ or $p_{k-1}\equiv 1\pmod{5}$, both of which are impossible (the first, since $p_{k-1}\geq 5$, and the second, from using again Theorem 1.1 of \cite{Benn}).

As previously,
\begin{align*}
\ord_5 c_n=-n-\ord_5(n!), \; \; \; 5^{n+\left[\frac{n}{4}\right]}c_n\in \mathbb{Z} \setminus \{ 0 \},
\end{align*}
for each $n \geq 0$, and
\begin{equation} \label{cee-big}
|c_n| < \left( \left[ \frac{n}{p_k} \right] +1 \right)  \left( \left[ \frac{n}{p_k^2} \right] +1 \right)  \left( \left[ \frac{n}{p_{k-1} p_k^2} \right] +1 \right)
(n+1).
\end{equation}
We have that
$$
w=p_k^{(a+2b+2)/5} y_1^a y_2^b y_3 =\pm X^{N/5}\sum_{n=0}^\infty c_nx^{-n}
$$
and hence writing
$$
P(X)=5^{N/5 + \left[ N/20 \right]}x^{N/5}\sum_{n=0}^{N/5} c_nX^{-n},
$$
that $5^{N/5+ \left[ N/20 \right]}w \mp P(X)$ is a nonzero integer.
Writing $m=N/5+1$, we thus have
$$
1 \leq 5^{N/4}|X|^{N/5}\sum_{n=m}^{\infty} (\left[ n/p_k\right]+1)(\left[ n/p_k^2\right]+1) (\left[ n/p_{k-1} p_k^2\right]+1) (n+1) |X|^{-n}.
$$
Writing $T_n=(\left[ n/p_k\right]+1)(\left[ n/p_k^2\right]+1) (\left[ n/p_{k-1} p_k^2\right]+1) (n+1)$, we have that 
$$
T_{n+1}/T_n \leq 106272/42875  < 3,
$$
provided $n \geq 3$ and $5 \leq p_{k-1} < p_k$ (with equality in the first inequality if $n=244, p_{k-1}=5$ and $p_k=7$). We thus have that
$$
1 \leq 5^{N/4}(m/p_k+1)(m/p_k^2+1)(m/p_{k-1} p_k^2+1)(m+1) |X|^{-1}\sum_{n=0}^{\infty} 3^n |X|^{-n},
$$
whence
$$
|X|<3+5^{N/4}(m/p_k+1)(m/p_k^2+1)(m/p_{k-1} p_k^2+1)(m+1).
$$
Since
$$
N \leq 4 p_k + 4 p_k^2 + p_{k-1} p_k^2-9\leq 2p_k^3-5,
$$
we have
$$
 |X| < 3 + 5^{\frac{1}{2} p_k^3} \left( \frac{2}{5} p_k^2 + 1 \right) \left( \frac{2}{5} p_k + 1 \right) 
\cdot 3 \cdot \left( \frac{2}{5} p_k^3 + 1 \right).
$$
On the other hand,
$$
|X| =|x|^{p_k^3} \geq 10^{4 p_k^3},
$$
a contradiction, since $p_k \geq 7$.
Our conclusion is thus that
\begin{equation} \label{case3}
n=p^5  \; \mbox{ for } \; p \; \mbox{ prime, } q=5 \; \mbox{ and } x \equiv 1 \mod{p}.
\end{equation}

To finish the proof of Theorem \ref{mtheorem}, it remains to treat case (\ref{case3}); we are currently unable to do so without appeal to the results of \cite{Mih}.
If we have a solution to
$$
\frac{x^{p^5}-1}{x-1} = y^5
$$
in integers $x$ and $y$ and prime $p$ with $|x|>1$ and $x \equiv 1 \mod{p}$, then
$$
\frac{x^{p^k}-1}{x^{p^{k-1}}-1} = p \, y_k^5
$$
for integers $y_k$, $1 \leq k \leq 5$. Appealing to Theorem 2 of Mih{\u{a}}ilescu \cite{Mih} and the fact that $|x| > 10^4$, it follows that
$$
10^{4 p^4} < |x|^{p^4} < 5^{10 p^2},
$$
a contradiction which completes the proof of Theorem \ref{mtheorem}.


\end{document}